\documentclass[sn-mathphys,Numbered]{sn-jnl}
\usepackage{amssymb,amsmath,amsfonts,amsthm,graphicx,verbatim,subcaption,tikz,pgfplots,float}
\usepackage{mathrsfs}%
\usepackage{xcolor}%
\usepackage{textcomp}%
\usepackage{manyfoot}%
\usepackage[utf8]{inputenc}
\usepackage{booktabs}       
\usepackage{nicefrac}       
\usepackage{microtype}      
\usepackage{thm-restate}
\usepackage{upgreek}
\newtheorem{example}{Example}
\newtheorem{theorem}{Theorem}

\newtheorem{remark}[theorem]{Remark}
\newtheorem{corollary}[theorem]{Corollary}
\newtheorem{lemma}[theorem]{Lemma}
\newtheorem{Proposition}[theorem]{Proposition}

\usepackage[capitalise,nameinlink,noabbrev]{cleveref}
\crefname{equation}{}{}
\crefname{Assumption}{Assumption}{Assumptions}
\crefname{Proposition}{Proposition}{Propositions}
\crefname{Definition}{Definition}{Definitions}
\newcommand{\ie}{i.e., }

\newcommand{\Ttran}{\mathsf{T}}

\newcommand{\de}{\,\mathrm{d}}

\newcommand{\defi}{:=}
\DeclareMathOperator{\grad}{grad}
\newcommand{\lowp}{\mathbf{u}_{l}}
\newcommand{\Sm}{\mathbb{S}^{n-1}}
\newcommand{\normml}{\left\vert\kern-0.25ex\left\vert\kern-0.25ex\left\vert}  
\newcommand{\normmr}{\right\vert\kern-0.25ex\right\vert\kern-0.25ex\right\vert}
\newcommand{\order}{\mathcal{O}}

\newcommand{\R}{\mathbb{R}}
\DeclareMathOperator{\dist}{dist}

\providecommand{\abs}[1]{\lvert#1\rvert}
\providecommand{\norm}[1]{\lVert#1\rVert}
\providecommand{\dual}[1]{\langle#1\rangle}

\title[PINVIT with improved convergence guarantee]{A preconditioned inverse iteration with an improved convergence guarantee}

\author[1]{\fnm{Foivos} \sur{Alimisis}}\email{foivos.alimisis@unige.ch}
\author[2]{\fnm{Daniel} \sur{Kressner}}\email{daniel.kressner@epfl.ch}
\author*[2]{\fnm{Nian} \sur{Shao}}\email{nian.shao@epfl.ch}
\author[1]{\fnm{Bart} \sur{Vandereycken}}\email{bart.vandereycken@unige.ch}

\affil[1]{\orgdiv{Section of Mathematics}, \orgname{University of Geneva} \orgaddress{\city{Geneva}, \postcode{1205}, \country{Switzerland}}}
\affil*[2]{\orgdiv{Institute of Mathematics}, \orgname{EPFL}, \orgaddress{\city{Lausanne}, \postcode{1015},  \country{Switzerland}}}
\begin{document}

\abstract{
Preconditioned eigenvalue solvers offer the possibility to incorporate preconditioners for the solution of large-scale eigenvalue problems, as they arise from the discretization of partial differential equations. The convergence analysis of such methods is intricate. Even for the relatively simple preconditioned inverse iteration (PINVIT),
which targets the smallest eigenvalue of a symmetric positive definite matrix, the celebrated analysis by Neymeyr is highly nontrivial and only yields convergence if the starting vector is fairly close to the desired eigenvector. In this work, we prove a new non-asymptotic convergence result for a variant of PINVIT. Our proof proceeds by analyzing an equivalent Riemannian steepest descent method and leveraging convexity-like properties. We show a convergence rate that nearly matches the one of PINVIT. As a major benefit, we require a condition on the starting vector that tends to be less stringent. This improved global convergence property is demonstrated for two classes of preconditioners with theoretical bounds and a range of numerical experiments.
}
\maketitle

\section{Introduction} \label{sec:intro}

Given a large-scale, symmetric positive definite (SPD) matrix $A \in \R^{n \times n}$ with eigenvalues 
$0 < \lambda_1<\lambda_2\leq \dotsb\leq \lambda_n$, this work considers the task of approximating the smallest eigenvalue $\lambda_1$ and an associated eigenvector $u_*$. Inverse iteration~\cite[Sec. 8.2.2]{Golub2013} addresses this task by applying the power method to the inverse:  $u_{t+1} = A^{-1} u_{t}$, combined with some normalization to avoid numerical issues.
This iteration inherits the excellent global convergence guarantee of the power method~\cite[Thm. 8.2.1]{Golub2013}: For \emph{almost every} choice of starting vector $u_{0}$, the angle between $u_*$ and $u_{t}$ converges linearly to zero with rate $\lambda_1/\lambda_2$.
Moreover, the Rayleigh quotient $\lambda(u_{t}):=u_{t}^\Ttran A u_{t} / u_{t}^\Ttran u_{t}$ converges linearly to $\lambda_1$ with rate $\lambda^2_1/\lambda^2_2$. 

A major limitation of the inverse iteration is that it requires to solve a linear system with $A$ in every iteration. Using, for example, a sparse Cholesky factorization of $A$ for this purpose may become expensive unless $A$ has a favorable sparsity pattern. In many situations, it is much cheaper to apply $B^{-1}$ instead of $A^{-1}$ for a preconditioner $B$ constructed, for example, from multigrid methods~\cite{Briggs2000}, domain decomposition~\cite{Toselli2005} or spectral sparsification~\cite{Kyng2016}. In principle, the availability of a preconditioner allows for the use of an iterative solver, such as the preconditioned conjugate gradient method~\cite[Sec. 11.5.2]{Golub2013}, for solving the linear systems with $A$ within inverse iteration. However, combining iterative methods in such an inner-outer iteration typically incurs redundancies. Instead, it is preferable to incorporate the preconditioner more directly, in a \emph{preconditioned eigenvalue solver}.

The fruitfly of preconditioned eigenvalue solvers is the Preconditioned INVerse ITeration (PINVIT)
\begin{equation} \label{eq:pinvit}
 u_{t+1} = u_{t} - B^{-1} r_t \quad \text{with} \quad r_t = A u_{t} - \lambda(u_{t}) u_{t}.
\end{equation}
While PINVIT can be viewed as a preconditioned steepest descent method for minimizing the Rayleigh quotient,
Neymeyr's seminal (non-asymptotic) convergence analysis~\cite{neymeyr2001geometric,Neymeyr2001b} is based on interpreting~\eqref{eq:pinvit} as a perturbed inverse iteration. \emph{Assuming} that $\lambda(u_{t}) \in [\lambda_1,\lambda_2)$, a convergence result by Knyazev and Neymeyr~\cite[Thm. 1]{knyazev2003geometric} states that
\begin{equation} \label{eq:convresult}
 \frac{\lambda(u_{t+1})-\lambda_1}{\lambda_2-\lambda(u_{t+1})} \le \rho^2 \frac{\lambda(u_{t})-\lambda_1}{\lambda_2-\lambda(u_{t})}, 
\end{equation}
with the convergence rate determined by
$\rho = 1-(1-\rho_B)(1-\lambda_{1}/\lambda_{2})$,
where $\rho_B$ is such that
\begin{equation}\label{eq:normalization PINVIT}
\| I - B^{-1} A \|_A \leq \rho_B < 1.
\end{equation}
Here and in the following, $\|\cdot\|_C$ denotes the vector and operator norms induced by an SPD matrix $C$.
If $\rho_B \ll 1$, this result shows  that PINVIT nearly attains the convergence rate of inverse iteration.
In principle, the condition~\eqref{eq:normalization PINVIT} can always be satisfied for any SPD matrix $B$ by suitably rescaling $B$ to $B / \eta$, which is equivalent to adding a step size $\eta>0$ to PINVIT: $u_{t+1} = u_{t} - \eta B^{-1} r_t$. With PINVIT being one of the simplest preconditioned eigenvalue solvers, its analysis also
provides important insights into the performance of more advanced methods like LOBPCG~\cite{knyazev2001toward}
and PRIMME~\cite{Stathopoulos2010}.
Recently, provable accelerations of PINVIT, in the sense of Nesterov's accelerated gradient descent~\cite{Nesterov1983}, have been introduced in~\cite{shao2023riemannian,shao2024epic}, based on certain convexity structures of the Rayleigh quotient.
The analysis of these methods requires conditions on the initial vector that are even stricter than the one required for PINVIT.

If $\lambda(u_{0}) \in [\lambda_1,\lambda_2)$ then~\eqref{eq:pinvit} implies that $\lambda(u_{t}) \in [\lambda_1,\lambda_2)$ is satisfied for all subsequent iterates $u_{t}$ of PINVIT and $u_{t}$ converges to $u_*$ (in terms of angles). 
However, this assumption on the initial vector $u_0$ is quite restrictive. In fact, for a Gaussian random initial $u_0$, the probability of achieving $\lambda(u_{0}) < \lambda_2$ quickly vanishes for larger $n$ and does not benefit from the quality of the preconditioner $B$. This is in stark contrast to both, inverse iteration ($B = A$) and steepest descent~\cite{alimisis2021distributed} ($B = I$), which converge to $u_{*}$ almost surely for a Gaussian random initial vector.

In this work, we attain a new non-asymptotic convergence result for a slight variation of PINVIT. For this purpose, we first reformulate the task
of computing the smallest eigenvalue and eigenvector as an equivalent Riemannian optimization problem on the unit sphere $\Sm$ in $\R^n$, with the preconditioner $B$ incorporated.
A similar but different reformulation was used in~\cite{shao2023riemannian}.
We show that standard Riemannian steepest descent~\cite[Algorithm~3.1]{Smith1994} applied to this problem coincides with a variant of PINVIT~\eqref{eq:pinvit} that uses a different step size and normalization. Inspired by~\cite{alimisis2021distributed}, we then establish and leverage a weaker notion of convexity (called weak-quasi-strong-convexity) to prove our main result,~\Cref{thm:main}: Riemannian steepest descent and, hence, our variant of PINVIT converge if the initial vector $u_0$ satisfies 
\begin{equation}
    \label{eq:iniConAD}
    \frac{{u_{0}^{\Ttran}Bu_{*}}}{\norm{u_{0}}_{B}\norm{u_{*}}_{B}} > \cos\varphi,
\end{equation}
where $\varphi$ is an angle measuring the distortion of the Euclidean geometry induced by the preconditioner at $u_*$; see \cref{defvarphi} for the precise definition.
The convergence is linear and we prove an asymptotic convergence rate that matches~\eqref{eq:convresult} up to a small factor.

For $B = I$ and $B = A$, it holds that $\cos\varphi =0$ and, thus, the condition~\eqref{eq:iniConAD} recovers the excellent global convergence properties of steepest descent and inverse iteration mentioned above. The practical use of PINVIT is between these two extreme scenarios and in such cases our numerical results indicate that the condition \cref{eq:iniConAD} is less stringent than $\lambda(u_{0}) < \lambda_2$. For the specific choices of mixed-precision and domain decomposition preconditioners, we provide theoretical results underlining that good preconditioners lead to $\cos\varphi \approx 0$.

\section{PINVIT as steepest descent} \label{sec:steepestdescent}

The results of this work are based on a novel formulation of PINVIT as (Riemannian) steepest descent on $\Sm$.
For this purpose, we define the following optimization problem for SPD matrices $A, B \in \mathbb{R}^{n \times n}$:
\begin{equation}
\label{main_problem}
    \min_{x \in \mathbb{S}^{n-1}} f(x),\qquad f(x):= - \frac{x^{\Ttran} B^{-1} x}{x^{\Ttran} B^{-1/2} A B^{-1/2} x}.
\end{equation}
Using the substitution $u = B^{-1/2} x$, we have that 
$$f(x) = - \frac{u^{\Ttran}u}{u^{\Ttran}Au}.$$
The minimum of $f$ is hence $-1/\lambda_1$ and is attained at $x_{*}=\frac{B^{1/2} u_*}{\|B^{1/2}u_*\|}$ for an eigenvector $u_*$ belonging to the eigenvalue $\lambda_1$ of $A$, where $\|\cdot\|$ denotes the Euclidean norm.

The formulation~\eqref{main_problem} is inspired by the previous work~\cite{shao2023riemannian}, which considers the minimization of $-1/f(x)$ instead of $f(x)$.
These two optimization problems are clearly equivalent and behave very similarly close to the optimum $x_{*}$. For a local convergence analysis, as the one performed in~\cite{shao2023riemannian}, the choice between the two optimization problems does not make a significant difference. For attaining results of a more global nature, this choice matters and it turns out that our new formulation~\cref{main_problem} is more suitable.
{\begin{remark} \label{remark:gevp}
    Our work also applies to generalized eigenvalue problems of the form $A - \lambda M$, for SPD matrices $A$ and $M$. The additional matrix $M$ can be absorbed by setting $\widehat{A} = M^{-1/2}AM^{-1/2}$, $\widehat{B} = M^{-1/2}BM^{-1/2}$ and $\widehat{x} = \widehat{B}^{-1/2}M^{-1/2}B^{1/2}x$, and one obtains the same type of optimization problem~\cref{main_problem}, simply with $A$, $B$  and $x$ replaced by $\widehat{A}$, $\widehat{B}$ and $\widehat{x}/\norm{\widehat{x}}$.
\end{remark}}

We view $\mathbb{S}^{n-1}$ as a Riemannian submanifold of $\mathbb{R}^n$ with the restricted Euclidean metric. Minimizing~\eqref{main_problem} by the \emph{Riemannian steepest descent} method yields the recurrence
\begin{equation}
    \label{algoSD}
    x_{t+1}=\exp_{x_{t}}(-\eta_{t} \grad f(x_{t})),
\end{equation}
for an initial vector $x_0 \in \mathbb{S}^{n-1}$, where $\grad$ denotes the Riemannian gradient on the sphere and $\exp_{x_{t}}$ denotes the exponential map  at $x_{t}$ on $\mathbb{S}^{n-1}$ (see below for explicit formulas). 
We impose the natural restriction
\begin{equation}
    \label{eq:stepsizen}
    0<\eta_{t}< \frac{\pi}{2\norm{\grad f(x_{t})}}
\end{equation}
on the step size.

The following proposition shows that the recurrence~\cref{algoSD} is a variant of PINVIT~\eqref{eq:pinvit} that uses a different step size\footnote{Recall that PINVIT can use step size $1$ thanks to the normalization of the preconditioner implied by~\eqref{eq:normalization PINVIT}.} and normalization.

\begin{Proposition}
\label{propEqv}
Consider the iterates $x_t$ produced by the recurrence~\eqref{algoSD} with a step size satisfying~\eqref{eq:stepsizen}. Then the transformed vectors $u_{t}:=B^{-1/2}x_{t}$ satisfy the recurrence
    \begin{equation} \label{eq:modpinvit}
        u_{t+1} = \beta_{t+1}( u_{t}-\eta_{t}^{*}B^{-1}r_{t}),
    \end{equation}
    with a certain step size $\eta_{t}^{*}>0$, a normalization $\beta_{t+1}>0$ chosen such that $\norm{u_{t+1}}_{B}=1$, and the 
    residual $r_{t} = Au_{t}-\lambda(u_{t})u_{t}$.
\end{Proposition}
\begin{proof}
A direct calculation of the Euclidean gradient of $f$ shows
\begin{equation} \label{eq:nabla fx} 
  \nabla f(x_t) = -\frac{2(B^{-1}x_t+f(x_t) B^{-1/2} A B^{-1/2}x_t)}{\| A^{1/2} B^{-1/2} x_t\|^2}.
\end{equation}
Because $I-x_t x_t^{\Ttran} $ is the orthogonal projection to the tangent space of the sphere at $x_t$, the Riemannian gradient is given by (see, e.g., \cite[Example 3.6.1]{absil2009optimization})
\begin{equation} \label{eq:expgrad}
    \grad f(x_t) = (I-x_t x_t^{\Ttran}) \nabla f(x_t) = \nabla f(x_t),
\end{equation}
where the latter equality follows from $x_t^{\Ttran} \nabla f(x_t) = 0$. In particular, this implies
that $\grad f(x_t)$ is zero if and only if the residual $r_{t}$ is zero. In this case, the recurrence~\eqref{eq:modpinvit} holds trivially. We may therefore assume $\grad f(x_t) \not= 0$ in the following. 

Using the explicit expression of the exponential map on the sphere from~\cite[Example 5.4.1]{absil2009optimization}, the recurrence~\eqref{algoSD} is rewritten as
\begin{equation*}
\begin{aligned}
    x_{t+1} &= \cos(\norm{\eta_{t}\grad f(x_{t})})x_{t}-\sin(\norm{\eta_{t}\grad f(x_{t})})\frac{\grad f(x_{t})}{\norm{\grad f(x_{t})}}\\ 
    &= \beta_{t+1}\bigl(x_{t}-\overline{\eta}_{t}\grad f(x_{t})\bigr),
\end{aligned}
\end{equation*}
where we set
\begin{equation*}
    \beta_{t+1} := \cos(\norm{\eta_{t}\grad f(x_{t})})
    \quad\text{and}\quad 
    \overline{\eta}_{t} := \frac{\tan(\norm{\eta_{t}\grad f(x_{t})})}{\norm{\grad f(x_{t})}}.
\end{equation*}
By~\eqref{eq:stepsizen}, $\overline{\eta}_{t}$ is well defined and positive.
Substituting $u_{t}=B^{-1/2}x_{t}$ and using~\cref{eq:nabla fx,eq:expgrad}, we obtain that
\begin{equation*}
\begin{aligned}
    u_{t+1} &= \beta_{t+1}\Bigl(u_{t}+\frac{2\overline{\eta}_{t}}{\norm{A^{1/2}B^{-1/2} x_{t}}^{2}} \bigl(B^{-1} u_{t}+f(x_{t}) B^{-1} A u_{t}\bigr)\Bigr)\\ 
    &= \beta_{t+1}\biggl(u_{t}+\frac{2\overline{\eta}_{t}}{u_{t}^{\Ttran}Au_{t}} \Bigl(B^{-1} u_{t}-\frac{1}{\lambda(u_{t})} B^{-1} A u_{t}\Bigr)\biggr)\\ 
    &= \beta_{t+1}(u_{t}-\eta_{t}^{*}B^{-1}r_{t}),
\end{aligned}
\end{equation*}
with
\begin{equation*}
    \eta_{t}^{*} := \frac{2\overline{\eta}_{t}u_{t}^{\Ttran}u_{t}}{(u_{t}^{\Ttran}Au_{t})^{2}} = \frac{2\tan(\norm{\eta_{t}\grad f(x_{t})})u_{t}^{\Ttran}u_{t}}{\norm{\grad f(x_{t})}(u_{t}^{\Ttran}Au_{t})^{2}}>0.
\end{equation*}
By definition, $x_{t+1}$ is on the sphere and, therefore, it follows immediately that $\norm{u_{t+1}}_{B}=1$.
\end{proof}

\section{Quality of preconditioner} \label{sec:precond}

In this section, we discuss quantities that measure the quality of the preconditioner $B$ in the context of preconditioned eigenvalue solvers.

\subsection{Global: spectral equivalence} \label{sec:precondglobal}

For any SPD matrices $A,B$, there exist
constants $0 < \nu_{\min} \le  \nu_{\max}$ such that
\begin{equation}
    \label{defnu}
    \nu_{\min} (x^{\Ttran} B x) \leq x^{\Ttran} A x \leq \nu_{\max} (x^{\Ttran} B x), \quad \forall x,
\end{equation}
a property sometimes called spectral equivalence. Equivalently,
\begin{equation}
    \label{estnu}
    \nu_{\min}  \|x\|^2 \leq\norm{A^{1/2}B^{-1/2}x}^{2}\leq \nu_{\max} \|x\|^2,  \quad \forall x.
\end{equation}
The tightest bounds are obtained by  choosing
$\nu_{\min}$ and $\nu_{\max}$ as the smallest and largest eigenvalues of $AB^{-1}$, respectively. As we will see below, their ratio $\kappa_{\nu} :=\nu_{\max}/\nu_{\min}$ determines the convergence rate of PINVIT and other preconditioned eigenvalue solvers.

While~\eqref{defnu} can always be satisfied, ideally $\kappa_{\nu}$ is not too large. In particular, when $A$ arises from the  discretization of a partial differential equation, a good preconditioner $B$ keeps $\kappa_{\nu}$ bounded as the discretization is refined; see also Section~\ref{sec:dd_mix}.

The inequality~\eqref{defnu} only implies the condition~\eqref{eq:normalization PINVIT} required by the convergence (analysis) of PINVIT when $B$ is \emph{scaled} in a suitable manner.
According to~\cite{neymeyr2002posteriori}, preconditioning with $\eta B^{-1}$ instead of $B^{-1}$ with $\eta = 2/(\nu_{\max} + \nu_{\min})$ leads to $\rho_B = (\kappa_{\nu} - 1) / (\kappa_{\nu} + 1) < 1$ in \eqref{eq:normalization PINVIT}.

\subsection{Local: angle of distortion}

Our condition on the initial vector will be based on an \emph{angle of distortion $\varphi$}, which measures the distortion induced by the preconditioner at
the eigenvector $u_{*}$:
    \begin{equation}
        \label{defvarphi}
        \varphi \defi \arcsin \frac{\norm{u_{*}}^{2}}{\norm{u_{*}}_{B}\norm{u_{*}}_{B^{-1}}} \in (0,\pi/2].
    \end{equation}
For the vector $x_{*} = \frac{B^{1/2}u_{*}}{\| B^{1/2} u_{*} \|}$, we have that
    \begin{equation*}
        \frac{x_{*}^{\Ttran}B^{-1}x_{*}}{\norm{x_{*}}\norm{B^{-1}x_{*}}} = \frac{\norm{u_{*}}^{2}}{\norm{u_{*}}_{B}\norm{u_{*}}_{B^{-1}}} = \sin\varphi.
    \end{equation*}
In other words, $\varphi$ is complementary to the angle between $x_{*}$ and $B^{-1}x_{*}$, as illustrated in~\cref{figConAngle}. 
\begin{figure}[htbp]
    \centering
    \begin{tikzpicture}
    \draw[thick] (3,0) arc[start angle=0, end angle=180, radius=3cm];  
    \draw[thick] (-3,0) -- (3,0);
    \draw[thick,->] (0,0) -- (0,3); 
    \node at (0.4,2.7) {$x_{*}$};  
    \draw[thick] (0,0) -- (3*0.9397, 3*0.3420);
    \draw[thick] (0,0) -- (-3*0.9397, 3*0.3420);
    \draw[thick,->] (0,0) -- (-2*0.3420,2*0.9397);
    \node at (-0.8,2.2) {$B^{-1}x_{*}$};
    \node at (1,1.1) {$\varphi$};
    \draw[thick,->] (0.9397,0.3420) arc[start angle=20, end angle=90, radius=1cm];
    \draw (0.9397/3,0.3420/3) -- (0.5977/3,1.2817/3) -- (-0.3420/3,0.9397/3) -- (0,0) -- cycle;
    \fill[gray!50] (3,0) arc[start angle=0, end angle=20, radius=3cm] -- (0,0) -- cycle;
    \fill[gray!50] (-3,0) arc[start angle=180, end angle=160, radius=3cm] -- (0,0) -- cycle;
\end{tikzpicture}
\caption{Angle of distortion $\varphi$. Vectors $x$ in the white region satisfy $\dist(x,x_{*})<\varphi$.}
\label{figConAngle}
\end{figure}
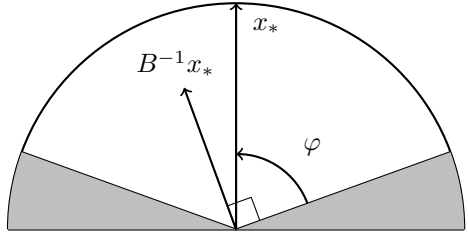

For $x\in\Sm$, we let 
\[
 \dist(x,x_{*}) :=\arccos(x^{\Ttran}x_{*})
\]
denote the angle between $x_{*}$ and $x$,
which happens to be the intrinsic Riemannian distance on the sphere. By suitably choosing the sign of $x_*$, we may always assume that $\dist(x,x_{*}) \in [0,\pi/2]$.
When $\dist(x,x_{*})<\varphi$, the following lemma establishes a lower bound on $x^{\Ttran}B^{-1}x_{*}$ that will play an important role for the so-called weak-quasi-convexity property of the function $f$ defined in~\eqref{main_problem}; see \cref{prop:weak_conv_prec} below.

    \begin{lemma}
\label{prop:local}
With the notation introduced above, we have that 
    \begin{equation*}
        x^{\Ttran}B^{-1}x_{*}\geq \frac{\norm{u_{*}}_{B^{-1}}^{2}}{\norm{u_{*}}^{2}}\Bigl(\cos\bigl( \dist(x,x_{*}) \bigr)-\cos\varphi\Bigr).
    \end{equation*}
    holds for any $x \in \Sm$ with $\dist(x,x_{*}) <\varphi$.
\end{lemma}
\begin{proof}
    Set $\sigma:=\norm{u_{*}}_{B^{-1}}^{2}/\norm{u_{*}}^{2}$. By the Cauchy--Schwarz inequality,
    \begin{equation*}
        x^{\Ttran}B^{-1}x_{*} = \sigma x^{\Ttran}x_{*} + x^{\Ttran}(B^{-1}-\sigma I)x_{*} \geq \sigma x^{\Ttran}x_{*} - \norm{(B^{-1}-\sigma I)x_{*}} .
    \end{equation*}
    On the other hand, it holds that
    \begin{equation*}
        \frac{\norm{(B^{-1}-\sigma I)x_{*}}^{2}}{\sigma^{2}} = \frac{\norm{B^{-1/2}u_{*}-\sigma B^{1/2}u_{*}}^{2}}{\sigma^{2}\norm{u_{*}}_{B}^{2}} = 1-\frac{\norm{u_{*}}^{4}}{\norm{u_{*}}_{B}^{2}\norm{u_{*}}_{B^{-1}}^{2}}=\cos^{2}\varphi,
    \end{equation*}
    where the second equality follows by expanding the square.
    The result follows from combining the two relationships, using that $\cos( \dist(x,x_{*}) ) = x^\Ttran x_*$.
\end{proof}

{
In the absence of preconditioning, that is, $B=I$, the angle of distortion is $\varphi=\pi/2$ and the inequality of Lemma~\ref{prop:local} becomes a trivial equality.
The same holds when $B\not=I$ in the (unrealistic) scenario that $u_{*}$ is also an eigenvector of $B$, which in particular holds when $B = A$. In the general case, one expects that $\varphi$ is still close to $\pi/2$ or, equivalently, $\cos \varphi \approx 0$ for a good preconditioner $B$. From~\cref{defnu}, one immediately obtains the bound 
\begin{equation}
\label{boundangle}
    \cos^{2}\varphi \leq 1-\kappa_{\nu}^{-1}.
\end{equation}
However, this bound is often not sharp and we will establish much tighter bounds for specific preconditioners in Section \ref{sec:numerics}.

The following lemma provides a useful variational representation of $\cos\varphi$.
\begin{lemma}
    \label{lemproj}
    The angle of distortion $\varphi$ satisfies
    \begin{equation}
        \label{defCosVarphi}
        \cos\varphi = \sup_{v^{\Ttran}u_{*}=0} \frac{{v^{\Ttran}B^{-1}u_{*}}}{\norm{v}_{B^{-1}}\norm{u_{*}}_{B^{-1}}}.
    \end{equation}
\end{lemma}
\begin{proof}
The supremum in~\eqref{defCosVarphi} is attained by the $B^{-1}$-orthogonal projection of $u_*$ onto the subspace $\mathrm{span}\{u_{*}\}^\perp$. This projection is given by the vector
\begin{equation}
\label{eq:defvs}
     v_{*} = u_{*}-\frac{\|u_*\|^2}{\|u_*\|^2_B}Bu_{*},
\end{equation}
which follows from verifying that $v_*^{\Ttran} u_* = 0$ and $v_*^{\Ttran} B^{-1} ( u_* - v_* ) = \frac{\|u_*\|^2}{\|u_*\|^2_B} v_*^{\Ttran} u_*  = 0$. 

Note that $u_{*}$, $v_{*}$ and $u_{*}-v_{*}$ form a right triangle with respect to the  $B^{-1}$-inner product, where $u_{*}$ is the hypotenuse. By Pythagoras, 
\begin{equation*}
    \frac{\abs{v_{*}^{\Ttran}B^{-1}u_{*}}^{2}}{\norm{v_{*}}_{B^{-1}}^{2}\norm{u_{*}}_{B^{-1}}^{2}} = 1-\frac{\abs{(u_{*}-v_{*})^{\Ttran}B^{-1}u_{*}}^{2}}{\norm{(u_{*}-v_{*})}_{B^{-1}}^{2}\norm{u_{*}}_{B^{-1}}^{2}}
    =1-\frac{\norm{u_{*}}^{4}}{\norm{u_{*}}_{B}^{2}\norm{u_{*}}_{B^{-1}}^{2}}=\cos^{2}\varphi,
\end{equation*}
where we use the definition \cref{eq:defvs} of $v_{*}$ in the second equality, and the definition \cref{defvarphi} of $\varphi$ in the third equality.
\end{proof}
\begin{remark} \label{remark:shao}
For the situation considered in this work, the definition of leading angle from~\cite[Definition~6]{shao2023riemannian} amounts to
    \begin{equation*}
        \vartheta \equiv \vartheta(I_n,-1/\lambda_{1}; f) = \inf_{f(x)\leq -1/\lambda_{1}}\inf_{v^{\Ttran}x=0}\arccos \Bigl(\frac{\abs{v^{\Ttran}B^{-1}x}}{\norm{v}_{B^{-1}}\norm{x}_{B^{-1}}}\Bigr).
    \end{equation*}
    Similar to the proof of~\cref{lemproj}, one can show that 
    \begin{equation*}
        \vartheta = \arcsin\frac{\norm{u_{*}}_{B}^{2}}{\norm{Bu_{*}}\norm{u_{*}}}.
    \end{equation*}
    Comparing with the definition of $\varphi$ in~\eqref{defvarphi}, one observes that
    both $\vartheta$ and $\varphi$ are angles between $Bu_{*}$ and $u_{*}$, one with respect to the standard Euclidean inner product and the other with respect to the $B^{-1}$-inner product.
\end{remark}
}

\section{Convergence analysis} \label{sec:convergence}

In this section, we study the convergence of the Riemannian steepest descent method~\eqref{algoSD} or, equivalently, 
the PINVIT-like method~\eqref{eq:modpinvit}. Our analysis utilizes concepts developed in~\cite{alimisis2021distributed, alimisis2022geodesic} for analyzing non-preconditioned eigenvalue solvers.
In particular, we will use~\emph{smoothness} and the so-called \emph{weak-quasi-strong-convexity} of the objective function $f$
defined in~\eqref{main_problem} to show that the distance of the iterates~\eqref{algoSD} to $x_*$ contracts linearly.

\subsection{Smoothness-type property}

Our analysis requires the following smoothness-type property, parametrized by a function $\gamma(x)>0$:
\begin{equation}\label{eq:smoothness_xstar}
f(x)-f_{*}  \geq \frac{1}{2 \gamma(x)} \norm{\grad  f(x)}^{2}, \quad \forall x\in\Sm,
\end{equation}
where $f_{*}:=f(x_{*})=-1/\lambda_{1}$ denotes the minimum of $f$. 
Note that standard smoothness in (convex) optimization implies \eqref{eq:smoothness_xstar}, but not vice versa.

    \begin{Proposition}
\label{prop:smoothness_prec}
The smoothness-type property~\eqref{eq:smoothness_xstar} holds with
\begin{equation*}
    \gamma(x)=\frac{\nu_{\max} \cdot (\lambda_{1}^{-1}-\lambda_{n}^{-1})}{\norm{A^{1/2} B^{-1/2} x}^{2}}.
\end{equation*}
\end{Proposition}
\begin{proof}
    Using the transformation
    \begin{equation}\label{eq:transformation_yx}
        y(x) := \frac{A^{1/2} B^{-1/2} x}{\norm{A^{1/2} B^{-1/2} x}} \in \Sm,
    \end{equation}
    we get
    \begin{equation}\label{eq:transformed_fbar}
    f(x)=\overline{f}(y(x)):=-y(x)^{\Ttran} A^{-1} y(x).
    \end{equation}
    with $\min\limits_{y\in\Sm} \overline{f}(y)=f_{*}=-1/\lambda_{1}$. Since $\overline{f}$ is the Rayleigh quotient for $-A^{-1}$, we can use the smoothness property derived in \cite{alimisis2021distributed}:
    \begin{equation}
    \label{bound_first}
    f(x)-f_{*} = \overline{f}(y(x))- f_{*} \geq \frac{1}{2 (\lambda_{1}^{-1}-\lambda_{n}^{-1})}\norm{\grad \overline{f}(y(x))}^{2} .
    \end{equation}
    It remains to phrase this property in terms of $x$ instead of $y$.

    By the chain rule, we have  $\de f(x)=\de \overline{f}(y(x)) \de y(x)$, which implies 
    \begin{equation}
        \label{bound_grads}
        \grad f(x) = \de y^{\Ttran}(x) \grad \overline{f}(y) 
        \quad\text{and}\quad 
        \norm{\grad f(x)}\leq \norm{\de y(x)}\norm{\grad \overline{f}(y)}.
    \end{equation}
To lower bound the right-hand side of~\eqref{bound_first}, we thus need to upper bound the spectral norm of $\de y(x)$.
Denote $C:=A^{1/2} B^{-1/2}$. A direct calculation shows that
\begin{equation*}
    \de y(x)v = \frac{C v}{\norm{C x}}- C x \frac{x^{\Ttran} C^{\Ttran} C v}{\norm{ C x}^{3}} = \frac{1}{\norm{C x}} \left( I -  \frac{C x x^{\Ttran} C^{\Ttran} }{\norm{ C x}^{2}} \right) C v
\end{equation*}
holds for any $v$. Taking the Euclidean norm and noticing that the matrix in parentheses is an orthogonal projector, we obtain
\begin{equation*}
    \norm{\de y(x)v}\leq \frac{\norm{C v}}{\norm{Cx}}\leq \frac{\sqrt{\nu_{\max}} \norm{v}}{\norm{A^{1/2} B^{-1/2} x}}
\end{equation*}
and, hence, $\| \de y(x)\| \le \sqrt{\nu_{\max}} / \norm{A^{1/2} B^{-1/2} x}$. Plugging this inequality into~\cref{bound_grads}, we get
\begin{equation*}
    \norm{\grad \overline{f}(y(x))}^{2} \geq \frac{\norm{A^{1/2} B^{-1/2} x}^{2}}{\nu_{\max}} \norm{\grad f(x)}^{2}.
\end{equation*}
Together with the bound~\cref{bound_first}, this gives the desired inequality:
\begin{equation*}
    f(x)-f_{*} \geq \frac{\norm{A^{1/2} B^{-1/2} x}^{2}}{2 \nu_{\max} (\lambda_{1}^{-1}-\lambda_{n}^{-1})}\norm{\grad f(x)}^{2}. \qedhere
\end{equation*}
\end{proof}

It is worth noting that \cref{prop:quadratic_growth_prec} combined with~\cref{estnu} give the global bound
    \begin{equation} \label{eq:boundgamma}
        \gamma(x) \leq \kappa_{\nu} \cdot (\lambda_{1}^{-1}-\lambda_{n}^{-1}), \quad \forall x \in \Sm.
    \end{equation}

\subsection{Quadratic growth}

In this and the next section, we derive two properties of $f$ that correspond to weakened notions of strong convexity.
We recall that $\dist(x_1,x_2)$ denotes the angle between two vectors $x_1,x_2$. If $\|x_1\| = \|x_2\| = 1$, it follows from a simple geometrical argument that
\begin{equation} \label{eq:distances}
    \norm{x_{1}-x_{2}}\leq \dist(x_{1},x_{2})\leq \frac{\pi}{2}\norm{x_{1}-x_{2}}.
\end{equation}

    \begin{Proposition}
\label{prop:quadratic_growth_prec}
    The function $f$ satisfies 
    \begin{equation*}
        f(x)-f_{*} \geq \frac{\mu(x)}{2}\dist^{2}(x,x_{*}), \quad \forall x \in \Sm,
    \end{equation*}
    with
    \begin{equation*}
        \mu(x) := \frac{8\nu_{\min} \cdot (\lambda_{1}^{-1}-\lambda_{2}^{-1})\norm{u_{*}}_{B}}{\pi^{2}\norm{A^{1/2}B^{-1/2}x}\norm{u_{*}}_{A}}.
    \end{equation*}
\end{Proposition}
\begin{proof}
As in the proof of \cref{prop:smoothness_prec}, we apply the transformation $y(x)$ from~\eqref{eq:transformation_yx}  to obtain the transformed objective function $\overline f$ in~\eqref{eq:transformed_fbar}.
By the quadratic growth of $\overline f$ established in \cite{alimisis2021distributed}, we have 
\begin{equation}
\label{quad_growth_standard}
    f(x)-f_{*} = \overline f(y(x))-f_{*} \geq (\lambda_{1}^{-1}-\lambda_{2}^{-1}) \dist^2(y(x),u_{*}).
\end{equation}
It thus remains to lower bound $\dist(y(x),u_{*})$ in terms of $\dist(x,x_{*})$. For this purpose, we may assume $\|u_{*}\| = 1$ without loss of generality. 
We first rewrite
\[
y(x) - u_* = A^{1/2}B^{-1/2} z \quad \text{with} \quad z = \frac{x}{\norm{A^{1/2}B^{-1/2}x}}-B^{1/2}A^{-1/2}u_{*}.
\]
Using~\eqref{eq:distances}, we obtain that 
\begin{equation}
    \label{estyuz}
    \dist^2(y(x),u_{*})  \geq \norm{y(x)-u_{*}}^{2}  = \norm{ A^{1/2} B^{-1/2}z }^{2} \geq \nu_{\min}\norm{z}^{2}.
\end{equation}
Since $x_{*}=B^{1/2}u_{*}/\norm{u_{*}}_{B}$ and $A^{-1/2}u_* = \lambda_1^{-1/2} u_*$, we can also write
\begin{equation*}
    z = \frac{x}{\norm{A^{1/2}B^{-1/2}x}}-\frac{\norm{u_{*}}_{B}}{\norm{u_{*}}_{A}}x_{*},
\end{equation*}
For any $\alpha_{1},\,\alpha_{2}\in\R$ and $x_{1},\,x_{2} \in \Sm$, it holds that
\begin{equation*}
    \norm{\alpha_{1}x_{1}-\alpha_{2}x_{2}}^{2}=\alpha_{1}^{2}+\alpha_{2}^{2}-2\alpha_{1}\alpha_{2}x_{1}^{\Ttran}x_{2}\geq \alpha_{1}\alpha_{2}\norm{x_{1}-x_{2}}^{2}.
\end{equation*} 
Using~\eqref{eq:distances} once more, we can therefore bound
\begin{equation*}
    \norm{z}^{2}\geq \frac{\norm{u_{*}}_{B}}{\norm{A^{1/2}B^{-1/2}x}\norm{u_{*}}_{A}}\norm{x-x_{*}}^{2}\geq \frac{4\norm{u_{*}}_{B}}{\pi^{2}\norm{A^{1/2}B^{-1/2}x}\norm{u_{*}}_{A}}\dist^{2}(x,x_{*}).
\end{equation*}
Combined with~\cref{quad_growth_standard,estyuz}, we yields the inequality
\begin{equation*}
    f(x)-f_{*} \geq \frac{4\nu_{\min}(\lambda_{1}^{-1}-\lambda_{2}^{-1})\norm{u_{*}}_{B}}
    {\pi^{2}\norm{A^{1/2}B^{-1/2}x}\norm{u_{*}}_{A}}\dist^{2}(x,x_{*}),
\end{equation*}
which is the desired result.
\end{proof}

By \cref{estnu}, the quantity $\mu(x)$ of~\cref{prop:quadratic_growth_prec} admits the constant lower bound
    \begin{equation} \label{eq:mu}
        \mu(x) \geq \frac{8(\lambda_{1}^{-1}-\lambda_{2}^{-1})}{\pi^{2}\kappa_{\nu}} =: \mu_0, \quad \forall x \in \Sm.
    \end{equation}
This shows that $\mu_0$-strong convexity implies the quadratic growth established by~\cref{prop:quadratic_growth_prec}, with $\mu(x)$ replaced by the constant $\mu_0$. This constant features the key quantities in the classical convergence result~\eqref{eq:convresult}: the spectral gap of $A^{-1}$ measured by $\lambda_{1}^{-1}-\lambda_{2}^{-1}$ and the spectral equivalence~\eqref{defnu} of the preconditioner measured by $\kappa_\nu$. 

\subsection{Weak-quasi-convexity}

We now establish our second convexity-like property that is essential for the analysis of the Riemannian steepest descent method~\eqref{algoSD}.
It quantifies how proceeding in the direction of the negative spherical gradient pushes the iterates towards the optimum.  For this purpose, we recall
that $P_x := I-xx^{\Ttran}$ is the orthogonal projector on the tangent space of $\Sm$ at
$x \in \Sm$. The logarithmic map on $\Sm$ admits the explicit expression (see, for example, \cite[Example 10.20]{boumal2023introduction})
\begin{equation}\label{eq:log}
\log_x(x_*) = \dist(x,x_*) \, \frac{P_x x_*}{\|P_x x_*\|}, 
\end{equation}
provided that $\dist(x,x_*) < \pi/2$. This is the inverse of the exponential map on $\Sm$, that is, $\exp_x(\log_x(x_*)) = x_*$.

    \begin{Proposition}
\label{prop:weak_conv_prec}
Suppose that $x\in\Sm$ satisfies $\dist(x,x_*) < \varphi$ with the angle of distortion $\varphi$ defined in~\eqref{defvarphi}.
Then 
    \begin{equation}
        \label{eq:weak_conv_prec}
        \dual{\grad f(x), -\log_{x}(x_{*})} \geq 2a(x)\bigl(f(x)-f(x_{*})\bigr),
    \end{equation}
    where $\dual{\cdot,\cdot}$ denotes the Euclidean inner product, and 
    \begin{equation*}
        a(x) := \frac{\lambda_{1}\norm{u_{*}}_{B^{-1}}^{2}(\cos(\dist(x,x_*))-\cos\varphi)}{\norm{A^{1/2}B^{-1/2}x}^{2}\norm{u_{*}}^{2}}.
    \end{equation*}
\end{Proposition}
\begin{proof}
To simplify notation, we set $\theta_{x} := \dist(x,x_*)$.
Because $\| P_{x}x_{*} \|^2 = 1 - (x^{\Ttran} x_*)^2 = 1 - \cos^2 \theta_x$, we can write~\eqref{eq:log} as
\begin{equation*}
     \log_{x}(x_{*}) = \frac{\theta_{x}}{\sin \theta_{x}} P_{x}x_{*}.
\end{equation*}
As mentioned in the proof of~\cref{propEqv},  $\grad f(x) = P_x \nabla f(x) = \nabla f(x)$. We therefore get
\begin{equation*}
    \dual{\grad f(x), -\log_{x}(x_{*})} = -\frac{\theta_{x}}{\sin\theta_{x}}\dual{P_{x}\nabla f(x),P_{x}x_{*}} = -\frac{\theta_{x}}{\sin\theta_{x}}\dual{\nabla f(x),x_{*}}.
\end{equation*}
Using the expression~\eqref{eq:nabla fx} for $\nabla f(x)$,
$B^{-1/2}AB^{-1/2} x_{*} = \lambda_1 B^{-1} x_*$, and $f_{*}=-1/\lambda_{1}$, one gets
\begin{equation*}
    \dual{\nabla f(x),x_{*}} = -\frac{2\lambda_{1}x^{\Ttran}B^{-1}x_{*}}{\norm{A^{1/2}B^{-1/2}x}^{2}}\bigl(f(x)-f_*\bigr).
\end{equation*}
Combining the two equations above gives
\begin{equation*}
    \dual{\grad f(x) , -\log_{x}(x_{*})} = \frac{2\lambda_{1}\theta_{x}(x^{\Ttran}B^{-1}x_{*})}{\norm{A^{1/2}B^{-1/2}x}^{2}\sin\theta_{x}}\bigl(f(x)-f_*\bigr).
\end{equation*}
The result now follows from the bound on $x^{\Ttran}B^{-1}x_{*}$ established in \cref{prop:local}, additionally using that $\theta_x / \sin\theta_{x} \ge 1$.
\end{proof}

\begin{remark}
    \label{rmk_a}
    If $\cos(\dist(x,x_*)) \geq \cos\varphi+c\sin^{2}\varphi$ for some $0<c<1/2$, the factor $a(x)$ of~\cref{prop:weak_conv_prec} can be bounded by a constant:
    \begin{equation*}
        a(x) \geq  \frac{c\norm{u_{*}}_{A}^{2}}{\norm{A^{1/2}B^{-1/2}x}^{2}\norm{u_{*}}_{B}^{2}}\geq \frac{c}{\kappa_{\nu}},
    \end{equation*}
    This follows from~\cref{estnu}, \cref{defvarphi}, and $A u_* =\lambda_1 u_*$.
\end{remark}

\subsection{Weak-quasi-strong-convexity}

The quadratic growth and weak-quasi-convexity properties established above result in the following important property.
    \begin{Proposition}
\label{prop:weak-strong-conv}
The function $f$ defined in~\eqref{main_problem} satisfies
    \begin{equation*}
   f(x)-f_{*} \leq \frac{1}{a(x)} \dual{\grad f(x), -\log_x(x_{*})}- \frac{\mu(x)}{2} \dist^{2} (x,x_{*}),
\end{equation*}
for every $x \in \Sm$ satisfying $\dist(x,x_*)<\varphi$, with $\mu(x)$ and $a(x)$ defined in \cref{prop:quadratic_growth_prec,prop:weak_conv_prec}, respectively.
\end{Proposition}
\begin{proof}
    By~\cref{prop:quadratic_growth_prec,prop:weak_conv_prec},
\begin{equation*}
    \frac{\mu(x)}{2} \textnormal{dist}^2(x,x^*) \leq f(x)-f^* \leq \frac{1}{2 a(x)} \langle \textnormal{grad}f(x), -\log_x(x^*) \rangle.
\end{equation*}
Note that $\dist(x,x_*)<\varphi$ implies $a(x) > 0$.
Applying this inequality twice shows the desired result:
\begin{align*}
    f(x)-f^*  &\leq \frac{1}{2 a(x)} \langle \textnormal{grad}f(x), -\log_x(x^*) \rangle+\frac{\mu(x)}{2} \textnormal{dist}^2(x,x^*)-\frac{\mu(x)}{2} \textnormal{dist}^2(x,x^*) \\ & \leq \frac{1}{a(x)} \langle \textnormal{grad}f(x), -\log_x(x^*) \rangle -\frac{\mu(x)}{2} \textnormal{dist}^2(x,x^*). \qedhere
\end{align*}
\end{proof}

\subsection{Convergence analysis}

\Cref{thm:main} below contains the main theoretical result of this work, on the contraction of the 
error (measured in terms of the angles $\dist(x_t,x_{*})$) for the iterates produced by the Riemannian steepest descent method for~\cref{main_problem}.
The condition on the initial vector prominently features the angle of distortion $\varphi$ defined in~\eqref{defvarphi}, whereas the contraction rate 
involves the relative eigenvalue gap for $A^{-1}$ and the quantity $\kappa_{\nu}=\nu_{\max}/\nu_{\min}$ measuring the spectral equivalence~\eqref{defnu} of the preconditioner.
   
\begin{theorem}
    \label{thm:main}
    For an eigenvector $u_{*}$ associated with the smallest eigenvalue $\lambda_1$ and an SPD preconditioner $B$, let
    $x_{*} := B^{1/2}u_{*} / \norm{B^{1/2}u_{*}}$.
 Apply the Riemannian steepest descent method~\cref{algoSD} to the optimization problem \cref{main_problem}, with a starting vector $x_{0}\in\Sm$ such that
    \begin{equation}
        \label{initialcon}
        \dist(x_0,x_*)<\varphi,
    \end{equation}
    and a step size $\eta_{t}$ satisfying
    \begin{equation*}
        \eta_{t} \leq \frac{a(x_{t})}{\gamma(x_{t})} = \frac{\lambda_{1}\norm{u_{*}}_{B^{-1}}^{2}(\cos(\textnormal{dist}(x_t,x_*))-\cos\varphi)}{\nu_{\max}\norm{u_{*}}^{2}(\lambda_{1}^{-1}-\lambda_{n}^{-1})},
    \end{equation*}
    with $\gamma(x)$ and $a(x)$ defined in~\cref{prop:smoothness_prec,prop:weak_conv_prec}.
    Then the iterates $x_t$ produced by~\cref{algoSD} satisfy
    \begin{equation*}
        \dist^{2}(x_{t+1},x_{*}) \leq (1-\xi_t)  \dist^{2}(x_t,x_{*}),
    \end{equation*}
    where 
    $\xi_t : = \eta_t\mu(x_t)a(x_t)$ with $\mu(x)$ defined in \cref{prop:quadratic_growth_prec}, respectively.     
    When fixing the step size $\eta_t = a(x_t)/\gamma(x_t)$ we have 
    \begin{equation}
        \label{defdelta}
        \xi_t = \frac{8\lambda_{1}^{2}\norm{u_{*}}_{B}\norm{u_{*}}_{B^{-1}}^{4}}{\pi^{2}\norm{u_{*}}^{4}\norm{u_{*}}_{A}}
       \frac{(\cos(\textnormal{dist}(x_t,x_*))-\cos\varphi)^{2}}{\norm{A^{1/2}B^{-1/2}x_{t}}^{3}}
        \frac{\lambda_{1}^{-1}-\lambda_{2}^{-1}}{\kappa_{\nu}(\lambda_{1}^{-1}-\lambda_{n}^{-1})}
    \end{equation}
    bounded below by a positive constant, and $\dist(x_{t},x_{*})$ converges linearly to zero.
 \end{theorem}   
\begin{proof}
By~\cref{algoSD}, we have $\log_{x_{t}}(x_{t+1}) = -\eta_{t} \grad f(x_{t})$. 
Since $\dist(x,y) = \| \log_x(y)\|$, it follows that
\begin{equation} \label{eq:inequalityblubber}
    \begin{aligned}
        \dist^2(x_{t+1},x_{*}) &\leq \norm{-\eta_{t} \grad f(x_{t})-\log_{x_{t}}(x_{*})}^{2}  \\ & = \eta_{t}^{2} \norm{\grad f(x_{t})}^2 + \dist^2(x_{t},x_{*}) +2 \eta_{t}\dual{\grad f(x_{t}),\log_{x_{t}}(x_{*})}, 
    \end{aligned}
\end{equation}
where the inequality is a consequence~\cite[Proposition 16]{alimisis2021distributed} of the Rauch comparison theorem.
By \cref{prop:weak-strong-conv,prop:smoothness_prec}, we have
\begin{align*}
    \frac{1}{a(x_{t})} \dual{ \grad f(x_{t}),\log_{x_{t}}(x_{*})} &\leq f_{*}-f(x_{t})-\frac{\mu(x_{t})}{2} \dist^{2}(x_{t},x_{*}) \\
    &\leq -\frac{1}{2 \gamma(x_{t})} \norm{ \grad f(x_{t})}^{2}-\frac{\mu(x_{t})}{2} \dist^{2}(x_{t},x_{*}).
\end{align*}
Multiplying with $2 \eta_{t} a(x_{t})$ and using the hypothesis $\eta_{t} \leq a(x_{t}) / \gamma(x_{t})$, this gives
\begin{equation*}
    \begin{aligned}
        2 \eta_{t} \dual{\grad f(x_{t}),\log_{x_{t}}(x_{*})} &\leq -\frac{\eta_{t} a(x_{t})}{\gamma(x_{t})} \norm{ \grad f(x_{t}) }^{2}-\eta_{t}\mu(x_{t})  a(x_{t})  \dist^{2}(x_{t},x_{*}) \\ & \leq -\eta_{t}^{2} \norm{\grad f(x_{t})}^{2} -  \eta_{t}\mu(x_{t}) a(x_{t})  \dist^{2}(x_{t},x_{*}).
    \end{aligned}
\end{equation*}
Plugging this inequality into~\eqref{eq:inequalityblubber} proves the first part of the theorem:
\begin{equation*}
    \dist ^2(x_{t+1},x_{*}) \leq (1-\eta_{t}\mu(x_{t}) a(x_{t})) \dist^{2}(x_{t},x_{*}).
\end{equation*}
The expression~\eqref{defdelta} directly follows from the definitions of $a(x_t),\gamma(x_t),\mu(x_t)$, implying $\dist(x_t,x_{*}) \leq \dist(x_0,x_{*})$.
Finally, the claimed linear convergence can be concluded from the fact that $\xi_{t}$ admits the constant lower bound
\begin{equation*}
    \xi_{t}\geq \frac{8\lambda_{1}^{2}\norm{u_{*}}_{B}\norm{u_{*}}_{B^{-1}}^{4}}{\pi^{2}\norm{u_{*}}^{4}\norm{u_{*}}_{A}}
       \frac{(\cos(\textnormal{dist}(x_0,x_*))-\cos\varphi)^{2}}{\nu_{\max}^{3/2}}
        \frac{\lambda_{1}^{-1}-\lambda_{2}^{-1}}{\kappa_{\nu}(\lambda_{1}^{-1}-\lambda_{n}^{-1})}>0,
\end{equation*}
where we use $\dist(x_t,x_{*}) \leq \dist(x_0,x_{*})$ and the spectral equivalence~\cref{estnu}.
\end{proof}

By~\Cref{propEqv}, \Cref{thm:main} establishes an error contraction, with contraction rate $1-\xi_t$, also for the PINVIT-like method~\eqref{eq:modpinvit}, \emph{if} the step size restriction \cref{eq:stepsizen} is satisfied. Using the smoothness-type property \cref{eq:smoothness_xstar} and the weak-quasi-convexity property \cref{eq:weak_conv_prec},
it follows that
    \begin{equation*}
        \frac{a(x_{t})}{\gamma(x_{t})}\leq \frac{2a(x_{t})\bigl(f(x_{t})-f(x_{*})\bigr)}{\norm{\grad f(x_{t})}^{2}} \leq \frac{-\dual{\grad f(x),\log_{x_{t}}(x_{*})}}{\norm{\grad f(x_{t})}^{2}}< \frac{\pi}{2\norm{\grad f(x_{t})}},
    \end{equation*} 
where the last inequality uses that $\|\log_{x_{t}}(x_{*})\| = \dist(x_t,x_*) < \pi/2$ is implied by~\eqref{initialcon}. Hence, the step size restriction $\eta_t \le a(x_t) / \gamma(x_t)$ imposed by \Cref{thm:main} always implies~\cref{eq:stepsizen}. In terms of the PINVIT iterates $u_{t}=B^{-1/2}x_{t}$, the initial condition \cref{initialcon} takes the form
    \begin{equation} \label{rmkInitial}
        \dist(x_0,x_*) = \dist_{B}(u_{0},u_{*})\defi \arccos\Bigl(\frac{{u_{0}^{\Ttran}Bu_{*}}}{\norm{u_{0}}_{B}\norm{u_{*}}_{B}}\Bigr) < \varphi,
    \end{equation}
    where the sign of $u_*$ is chosen such that $u_{0}^{\Ttran}Bu_{*} \ge 0$.

The following corollary establishes convergence for a constant step size.

\begin{corollary}
\label{cor:non_asympt}
    If
    \begin{equation}
        \label{init_na}
       \cos \bigl(\dist(x_0,x_*)\bigr) \geq \cos\varphi+c\sin^{2}\varphi \hspace{2mm} \text{and} \hspace{2mm} \eta = \frac{c}{\kappa_{\nu}^{2}(\lambda_{1}^{-1}-\lambda_{n}^{-1})}    
       \end{equation}
     for some $0<c<1/2$, then
     the Riemannian steepest descent method~\cref{algoSD}
with step size $\eta$ produces iterates $x_t$ satisfying
     \begin{equation}
        \label{rate_na}
        \dist^{2}(x_t,x_*) \leq \left(1-\frac{8c^{2}(\lambda_{1}^{-1}-\lambda_{2}^{-1})}{\pi^{2}\kappa_{\nu}^{4}(\lambda_{1}^{-1}-\lambda_{n}^{-1})} \right)^t \textnormal{dist}^2(x_0,x_*).
     \end{equation}
     Thus, $x_t$ converges linearly to $x_*$.
\end{corollary}
\begin{proof}
    The proof proceeds by induction on $t$. The result for $t=0$ is trivial. Suppose \cref{rate_na} holds for some $t \ge 1$ and we now show that it also holds for $t+1$. From \cref{rate_na,init_na}, it follows that
    \begin{equation*}
        \cos\bigl(\dist(x_t,x_{*})\bigr)\geq \cos\bigl(\dist(x_0,x_{*})\bigr) \geq \cos\varphi+c\sin^{2}\varphi.
    \end{equation*} 
    As shown in~\cref{eq:boundgamma} and \cref{rmk_a}, we have  
    $\gamma(x_{t})\leq \kappa_{\nu}(\lambda_{1}^{-1}-\lambda_{n}^{-1})$ and 
    $a(x_{t})\geq c / \kappa_{\nu}$.
    Hence, the choice of $\eta$ in \cref{init_na} satisfies the condition $\eta\leq a(x_t)/\gamma(x_t)$. By \cref{thm:main}, 
    \begin{equation*}
        \dist^{2}(x_{t+1},x_*) \leq (1-\eta\mu(x_t)a(x_t)) \dist^{2}(x_t,x_*).
    \end{equation*}
    Using the lower bound~\eqref{eq:mu} on $\mu(x_t)$ and, once again,
    $a(x_{t})\geq c / \kappa_{\nu}$,
    the contraction rate can be bounded by 
    \begin{equation*}
        1-\eta\mu(x_t)a(x_t)\leq 1-\eta \frac{8c(\lambda_{1}^{-1}-\lambda_{2}^{-1})}{\pi^{2}\kappa_{\nu}^{2}} = 1-\frac{8c^{2}(\lambda_{1}^{-1}-\lambda_{2}^{-1})}{\pi^{2}\kappa_{\nu}^{4}(\lambda_{1}^{-1}-\lambda_{n}^{-1})}.
    \end{equation*}
    This completes the induction step. 
\end{proof}

Corollary \ref{cor:non_asympt} immediately yields a statement on the iteration complexity.
\begin{corollary}
Suppose that Riemannian steepest descent~\cref{algoSD} is
applied to~\ref{main_problem} with starting vector $x_0$ and step size $\eta$ satisfying~\eqref{init_na}.
Then an approximation $x_T$ of $x_*$ such that $\textnormal{dist}(x_T,x_*) \leq \epsilon$ is returned after 
   \begin{equation*}
       T = \mathcal{O} \left(\frac{\kappa_{\nu}^4}{c^2} \frac{\lambda_{1}^{-1}-\lambda_{n}^{-1}}{\lambda_{1}^{-1}-\lambda_{2}^{-1}} \log \frac{\textnormal{dist}(x_0,x_*)}{\epsilon} \right)
   \end{equation*}
iterations.
\end{corollary}

The following lemma simplifies the condition on the starting vector
in~\cref{cor:non_asympt}, at the expense of making it potentially (much) stricter.
\begin{lemma}
If
\[
 \cos^{2} (\textnormal{dist}(x_0,x_*)) \geq 1-\frac{1-2c}{\kappa_{\nu}}, \quad 0<c<1/2,
\]
then the condition~\eqref{init_na} on the starting vector $x_0$ is satisfied.
\end{lemma}
\begin{proof}
To establish the result, we show that 
$1-\frac{1-2c}{\kappa_{\nu}}\geq (\cos\varphi+c\sin^{2}\varphi)^{2}$
holds for every $0<c<1/2$.
For this purpose, consider the quadratic function 
    \begin{equation*}
        \begin{aligned}
            q(c) &= \bigl(\cos\varphi+c\sin^{2}\varphi\bigr)^{2}-1+(1-2c) / {\kappa_{\nu}}\\ 
            &= (\sin^{4}\varphi) c^{2}+2\bigl(\cos\varphi\sin^{2}\varphi-{1}/{\kappa_{\nu}}\bigr)c+{1}/{\kappa_{\nu}}-\sin^{2}\varphi.        
        \end{aligned}
    \end{equation*}
    By \cref{boundangle}, we know that
    $
        q(0) = {1} / {\kappa_{\nu}}-\sin^{2}\varphi\leq 0.
    $
    At the same time, we have
    \begin{equation*}
        q(1/2) = \frac14 {\sin^{4}\varphi} +\cos\varphi\sin^{2}\varphi-\sin^{2}\varphi \leq 0.
    \end{equation*}
    Because $q$ is quadratic with leading non-negative coefficient, it follows that
    $q(c)\leq 0$ for every $0< c< 1/2$, which completes the proof.
\end{proof}

We now derive the \emph{asymptotic} convergence rate implied by~\cref{thm:main}. This asymptotic rate is much more favorable than the non-asymptotic rate established in \cref{cor:non_asympt}.
\begin{Proposition} \label{prop:blabla}
{For the Riemannian steepest descent method~\cref{algoSD} with step size $\eta_t=a(x_t)/\gamma(x_t)$,}
the quantity 
$\xi_t$ determining the convergence rate $1 - \xi_t$, according to~\cref{thm:main},
satisfies
    \begin{equation*}
        \xi_{\infty}\defi \lim_{t\to\infty} \xi_t
        = \frac{8}{\pi^{2}(1+\cos\varphi)^{2}}
        \frac{\lambda_{1}^{-1}-\lambda_{2}^{-1}}{\kappa_{\nu}(\lambda_{1}^{-1}-\lambda_{n}^{-1})}.
    \end{equation*}
\end{Proposition}
\begin{proof}
{\cref{thm:main} shows that $\dist(x_{t},x_{*})\to0$ as $t\to\infty$.
    Inserted} into~\eqref{defdelta}, this gives
    \begin{equation*}
        \xi_{\infty} = \frac{8\lambda_{1}^{2}\norm{u_{*}}_{B}\norm{u_{*}}_{B^{-1}}^{4}}{\pi^{2}\norm{u_{*}}^{4}\norm{u_{*}}_{A}}
        \frac{(1-\cos\varphi)^{2}}{\norm{A^{1/2}B^{-1/2}x_{*}}^{3}}
        \frac{\lambda_{1}^{-1}-\lambda_{2}^{-1}}{\kappa_{\nu}(\lambda_{1}^{-1}-\lambda_{n}^{-1})}.
    \end{equation*}
    Using the relations
    \begin{equation*}
        \lambda_{1}^{2} = \frac{\norm{u_{*}}_{A}^{4}}{\norm{u_{*}}^{4}},\quad \norm{A^{1/2}B^{-1/2}x_{*}}^{3} = \frac{\norm{u_{*}}_{A}^{3}}{\norm{u_{*}}_{B}^{3}}\quad\text{and}\quad \sin\varphi = \frac{\norm{u_{*}}^{2}}{\norm{u_{*}}_{B}\norm{u_{*}}_{B^{-1}}},
    \end{equation*}
    the expression for $\xi_{\infty}$ simplifies to
    \begin{equation*}
        \xi_{\infty} = \frac{8(1-\cos\varphi)^{2}}{\pi^{2}\sin^{4}\varphi}
        \frac{\lambda_{1}^{-1}-\lambda_{2}^{-1}}{\kappa_{\nu}(\lambda_{1}^{-1}-\lambda_{n}^{-1})} = \frac{8}{\pi^{2}(1+\cos\varphi)^{2}}
        \frac{\lambda_{1}^{-1}-\lambda_{2}^{-1}}{\kappa_{\nu}(\lambda_{1}^{-1}-\lambda_{n}^{-1})}. \qedhere
    \end{equation*}     
\end{proof}

The convergence result~\eqref{eq:convresult} by Knyazev and Neymeyr shows that the \emph{eigenvalue} approximations of PINVIT converge linearly with the asymptotic convergence rate $\rho^2$. When $B$ is optimally scaled, then
\begin{equation*}
    \rho = 1-\frac{2}{\kappa_{\nu}+1}\frac{\lambda_{1}^{-1}-\lambda_{2}^{-1}}{\lambda_{1}^{-1}};
\end{equation*}
see \Cref{sec:precondglobal}. On the other hand, \Cref{prop:blabla} establishes the asymptotic convergence rate $1-\xi_\infty$ for the \emph{eigenvector} approximation error. As the eigenvalue approximation error is quadratic in the eigenvector approximation error (see, for example,~\cite[Eq. (27.3)]{trefethenbau}),
it is reasonable to compare $\xi_{\infty}$ with $1-\rho$:
\begin{equation*}
    \xi_{\infty} = (1-\rho)\cdot \frac{4}{\pi^{2}(1+\cos\varphi)^{2}}\cdot \frac{\kappa_{\nu}+1}{\kappa_{\nu}}\cdot (1+\lambda_{n}^{-1}).
\end{equation*}
Because $0\leq \cos\varphi< 1$, this shows that our asymptotic rate matches (up to a small constant) the sharp rate by Knyazev and Neymeyr.

\section{Distortion angle for specific preconditioners} \label{sec:dd_mix}

The convergence results of the previous section, most notably~\cref{thm:main},
requires the condition~\eqref{rmkInitial} on the initial vector $u_{0}$, which can be restated as
\begin{equation}
    \label{Initcond_dd_mix}
    \frac{{u_{0}^{\Ttran}Bu_{*}}}{\norm{u_{0}}_{B}\norm{u_{*}}_{B}} > \cos\varphi = \sup_{v^{\Ttran}u_{*}=0} \frac{v^{\Ttran}B^{-1}u_{*}}{\norm{v}_{B^{-1}}\norm{u_{*}}_{B^{-1}}},
\end{equation}
where $u_{*}$ is an eigenvector belonging to the smallest eigenvalue $\lambda_{1}$ of $A$. 
For a (Gaussian) random vector $u_0$, the left-hand side of~\cref{Initcond_dd_mix} is nonzero
almost surely, but it is unlikely to be far away from zero. Therefore, a good global convergence guarantee requires $\cos\varphi$ to be small.
In this section, we will demonstrate for two specific types of preconditioners 
that $\cos\varphi$ can be close to zero under reasonable assumptions. 

\subsection{Additive Schwarz preconditioners}
\label{sec:AS}

Domain decomposition methods (DDM) are widely used strategies for solving large-scale partial differential equations (PDEs). They are based on splitting a PDE, or an approximation of it, into coupled problems on smaller subdomains that collectively form a (possibly overlapping) partition of the original computational domain. A powerful way to analyze and develop DDM is through a subspace perspective~\cite{Xu1992} that divides the solution space into smaller subspaces, typically corresponding to the geometric structure of the subdomain partition. Here, we consider an additive Schwarz preconditioner as a representative DDM approach. Further details on DDM can be found in several classical references on the topic, such as \cite{Toselli2005}. The following discussion builds on the previous work \cite{shao2023riemannian}.

We first briefly describe a relatively standard mathematical setting for elliptic PDEs. 
On a convex polygonal domain $\Omega \subset \R^{d}$ with $d=2$ or $3$, consider a
symmetric and uniformly positive definite coefficient matrix $\{a_{ij}(x)\}_{i,\,j=1}^{d}$ such that $a_{ij}(x) \in C^{0,1}(\overline{\Omega})$ for $i, j = 1, \dotsc, d$. Let $V_{H} \subset V_{h} \subset H_{0}^{1}(\Omega)$ be continuous, piecewise linear finite element spaces based on quasi-uniform triangular partitions $\mathcal{T}_{H}$ and $\mathcal{T}_{h}$ of $\Omega$, such that $\mathcal{T}_{h}$ is a refinement of $\mathcal{T}_{H}$, and $0 < h < H < 1$ are the maximum mesh sizes of $\mathcal{T}_{h}$ and $\mathcal{T}_{H}$, respectively.
Then the elliptic PDE eigenvalue problem discretized on $V_h$ takes the following form:
\begin{equation}
    \label{lapevpfem}
    \mathcal{A}(u_{*}, v) = \lambda_{1}\dual{u_{*}, v}_{2} \quad \forall\, v \in V_{h}, \quad\text{where}\  \norm{u_{*}}_{2} = 1 \ \text{and} \ u_{*} \in V_{h}.
\end{equation}
Here $\dual{\cdot, \cdot}_{2}$ and $\norm{\cdot}_{2}$ denote the $L^{2}$ inner product and norm, respectively, and
\begin{equation}
    \mathcal{A}(u, v) \defi  \sum_{i, j = 1}^{d} \int_{\Omega} a_{ij}(x) \frac{\partial u}{\partial x_{i}} \frac{\partial v}{\partial x_{j}} \de x.
\end{equation}
The global solver is the linear operator $A^{-1}: V_h \to V_h$ such that
$u \mapsto A^{-1}u$ satisfies 
\begin{equation*}
    \mathcal{A}(A^{-1}u, v) = \dual{u, v}_{2} \quad \forall\, v \in V_{h}.
\end{equation*}
We are interested in an additive Schwarz preconditioner $B^{-1}$ for $A^{-1}$. 

To aid in understanding, we present a specific example of additive Schwarz preconditioners, following the structure outlined in \cite[Section~7.4]{Brenner2008}.

\begin{example}[Two-level overlapping domain decomposition preconditioner]
\label{exp_dd}
    Consider the region $\Omega = [0, 1]^2$. Let $\mathcal{T}_{H}$ be a coarse triangulation as shown in \cref{figDD}. The region $\Omega$ is divided into non-overlapping subdomains $\tilde{\Omega}_{j}$ for $1 \leq j \leq 16$, which are aligned with $\mathcal{T}_{H}$. Subsequently, $\mathcal{T}_{H}$ is further subdivided to obtain the finer triangulation $\mathcal{T}_{h}$.
    Define $\Omega_{j} = \tilde{\Omega}_{j, \delta} \cap \overline{\Omega}$, where $\tilde{\Omega}_{j, \delta}$ is an open set obtained by enlarging $\tilde{\Omega}_{j}$ by a band of width $\delta$, ensuring $\Omega_{j}$ is aligned with $\mathcal{T}_{h}$ as shown in \cref{figDD}. One often assumes that the overlapping ratio $\delta/H$ is bounded below by a constant, which is $0.5$ in this case.

    Let $V_{j}\subset V_{h}$ denote the subspace of continuous, piecewise linear functions supported in $\Omega_{j}$ for $1\leq j\leq 16$. 
    Define the coarse/local solvers $A_{H}^{-1}$ and $A_{j}^{-1}$ through
    \begin{equation*}
        \begin{aligned}
            \mathcal{A}(A_{H}^{-1}u_{H},v_{H})&=\dual{u_{H},v_{H}}_{2}\quad \forall\, v_{H} \in V_{H},\\ 
            \mathcal{A}(A_{j}^{-1}u_{j},v_{j})&=\dual{u_{j},v_{j}}_{2}\quad \forall\, v_{j} \in V_{j}.
        \end{aligned}
    \end{equation*}
    Then the two-level overlapping domain decomposition preconditioner is given by
    \begin{equation*}
        B^{-1} = I_{H}A_{H}^{-1}I_{H}^{\Ttran} + \sum_{j = 1}^{16} I_{j}A_{j}^{-1}I_{j}^{\Ttran},
    \end{equation*}
    where $I_{H}:\, V_{H}\mapsto V_{h}$ and $I_{j}:\, V_{j}\mapsto V_{h}$ are the natural injection operators, \ie $I_{H}v_{H}=v_{H}$ for all $v_{H}\in V_{H}$, and $I_{j}v_{j}=v_{j}$ for all $1\leq j\leq 16$ and $v_{j}\in V_{j}$.
    
    \begin{figure}[htbp]
        \centering
        \includegraphics[width=0.7\textwidth]{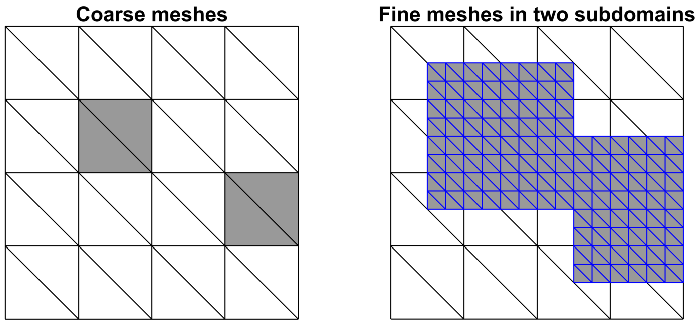}
        \caption{Construction of an overlapping domain decomposition. Example and figure taken from \cite[Example~30]{shao2023riemannian}.}
        \label{figDD}
    \end{figure}
\end{example}

Under reasonable assumptions, such as those stated in~\cite[Assumptions~2.2---2.4]{Toselli2005}, it holds that $\cos\varphi = \order(H)$ as $H\to 0$. To see this,
we employ the following results from \cite[Lemmas.~34~and~35]{shao2023riemannian}, which hold under such assumptions:
\begin{equation}
    \label{propB}
    \norm{B^{-1}u_{*} - \lambda_{H}^{-1}u_{*}}_{\mathcal{A}} \leq c_{d}\lambda_{1}^{-1/2} H
    \quad \text{and} \quad
    \norm{v}_{A^{-1}} \leq c_{d} \norm{v}_{B^{-1}} \quad \forall\, v \in V_{h},
\end{equation}
where $\lambda_{H}$ is the smallest eigenvalue of $\mathcal{A}(\cdot, \cdot)$ in $V_{H}$, $\norm{\cdot}_{\mathcal{A}}$, $\norm{\cdot}_{A^{-1}}$, and $\norm{\cdot}_{B^{-1}}$ are the norms induced by $\mathcal{A}(\cdot, \cdot)$, $A^{-1}$, and $B^{-1}$, respectively, and $c_{d}>0$
is a constant independent of the mesh sizes $h,H$.
For any $v\in V_{h}$ satisfying $\dual{u_{*},v}_{2}=0$ and $\norm{v}_{B^{-1}}=1$, the Cauchy--Schwarz inequality yields
\begin{equation*}
    \dual{B^{-1}u_{*}, v}_{2} =  \dual{B^{-1}u_{*} - \lambda_{H}^{-1}u_{*}, v}_{2} \leq  \norm{v}_{A^{-1}} \norm{B^{-1}u_{*} - \lambda_{H}^{-1}u_{*}}_{\mathcal{A}} \leq c_{d}^{2}  \lambda_{1}^{-1/2}H.
\end{equation*}
By the variational representation~\cref{Initcond_dd_mix} of $\varphi$, 
\begin{equation}
\label{est_dd}
    \cos\varphi = \sup_{\dual{u_{*}, v}_{2} = 0} \frac{\dual{B^{-1}u_{*}, v}_{2}}{\norm{v}_{B^{-1}} \norm{u_{*}}_{B^{-1}}} \leq \frac{c_{d}^{2} \lambda_{1}^{-1/2}H}{\norm{u_{*}}_{B^{-1}}} \leq c_{d}^{3}  H.
\end{equation}
As $c_{d}$ is independent of $h,H$,  it follows that $\cos\varphi = \order(H)$ as $H\to 0$.

\subsection{Mixed-precision preconditioners} \label{sec:mixed}

In this section, we study the condition \cref{Initcond_dd_mix} when using mixed-precision preconditioners as proposed in~\cite{Kressner2023}. For this purpose, we consider two levels of precision: a working precision and a lower precision, for example, IEEE double and single precision. The preconditioner is constructed in lower precision while the rest of the computations are carried out in working precision. For simplicity, the effects of round-off errors in working precision are ignored.

Consider the Cholesky factorization  $A = LL^{\Ttran}$, and let $\widehat{L}$ be the Cholesky factor computed in lower precision. We define the preconditioner $B$ as $B^{-1}x\defi \widehat{L}^{-\Ttran}(\widehat{L}^{-1}x)$, which is implemented by solving two triangular linear systems by performing forward and backward substitution in lower precision. 
By~\cite[Lemma~3]{Kressner2023}, $B^{-1}$ is a high-quality preconditioner for $A$, which satisfies 
\begin{equation*}
    \norm{I-A^{1/2}B^{-1}A^{1/2}}\leq \frac{\epsilon_{l}}{1-\epsilon_{l}}
\end{equation*}
where we assume $\epsilon_{l} := 4n(3n+1)(\lambda_{n}/\lambda_{1})\lowp< 1$ and $\lowp$ denotes unit roundoff in lower precision. Note that 
\begin{equation*}
    1-\norm{I-A^{1/2}B^{-1}A^{1/2}}\leq \lambda_{\min}(B^{-1}A)\leq \lambda_{\max}(B^{-1}A) \leq 1+\norm{I-A^{1/2}B^{-1}A^{1/2}}. 
\end{equation*}
Using the bound~\cref{boundangle} for $\cos\varphi$, it follows that
\begin{equation*}
    \cos^{2}\varphi \leq 1-\kappa_{\nu}^{-1} = 1-\frac{\lambda_{\min}(B^{-1}A)}{\lambda_{\max}(B^{-1}A)} \leq 1-\frac{1-\frac{\epsilon_{l}}{1-\epsilon_{l}}}{1+\frac{\epsilon_{l}}{1-\epsilon_{l}}} = 2\epsilon_{l}.
\end{equation*}
Usually $\epsilon_{l}\ll 1$ and, hence, $\cos\varphi\leq \sqrt{2\epsilon_{l}}$ is close to zero. This implies that a random starting vector nearly always 
satisfies the condition \cref{Initcond_dd_mix}. In contrast, the condition 
$\lambda(u_{0}) < \lambda_2$ required by the classical analysis of PINVIT does not enjoy any benefit from such a high-quality preconditioner.

\section{Numerical experiments} \label{sec:numerics}

In this section, we present some numerical experiments to provide insight into the behavior of $\varphi$ and a comparison between our initial condition~\cref{initialcon} and the classical condition $\lambda(u_{0}) \in [\lambda_1,\lambda_2)$.  All numerical experiments in this section have been implemented in Matlab 2022b and were carried out on an AMD Ryzen~9 6900HX Processor (8 cores, 3.3--4.9 GHz) and 32 GB of RAM.

\subsection{Laplace eigenvalue problems}

The experiments in this section target the smallest eigenvalue for the Laplacian eigenvalue problem with zero Dirichlet boundary condition on the unit square $\Omega = [0,1]^{2}$:
\begin{equation} \label{eq:laplace}
    \begin{aligned}
        -\Delta u &= \lambda u \quad \text{in }\Omega, \\ 
        u &= 0 \quad \text{on }\partial\Omega,
    \end{aligned}
\end{equation}
We will consider two different scenarios:
\begin{description}
    \item[AGMG] Five-points finite difference discretization of~\eqref{eq:laplace} on a regular grid of grid size $h$, together with an AGMG preconditioner,
    \item[DDM] Piecewise linear finite element discretization of~\eqref{eq:laplace} on a regular
    mesh of {mesh width $h$}, as shown in~\cref{exp_dd}, together with a   
    DDM preconditioner.
\end{description}

Detailed descriptions of AGMG (aggregation-based algebraic multigrid) preconditioners can be found in \cite{Notay2010,Napov2012}; we use the implementation from~\cite{NotayAGMG} (release 4.2.2). For DDM, we use the setting described in \cref{exp_dd}; a two-level overlapping domain decomposition preconditioner with an overlapping ratio of 0.5 is applied. {Note that in the latter case, we actually solve a generalized eigenvalue problem $A- \lambda M$, see Remark~\ref{remark:gevp}, with $M$ representing the mass matrix from the finite element method.}

\subsubsection{Behavior of $\varphi$}

The purpose of the first experiment is to study the angle of distortion $\varphi$. A small value of $\cos \varphi$ is favorable for our theory, because this implies that the condition on the initial vector becomes loose. We let $A = -\Delta_{h}$ denote the discretization of the Laplacian and $B$ denote the preconditioner. For either of the two scenarios described above, the preconditioner is only available implicitly, through matrix-vector products with $B^{-1}$.
The ratio $\kappa_{\nu}$ can be obtained by computing the smallest and largest eigenvalues of $-B^{-1}\Delta_{h}$ with the Lanczos method. The definition of the angle $\varphi$ requires the computation of both $Bu_{*}$ and $B^{-1}u_{*}$. While the second computation is straightforward, the first computation is not, because $B$ is not explicitly available. Instead of the matrix-vector multiplication $Bu_{*}$, we solve the linear system $B^{-1}z=u_{*}$ using the preconditioned conjugate gradient method with $-\Delta_{h}$ as the  preconditioner.

Defining
\begin{equation*}
    \chi := \frac{\cos^{2}\varphi}{1-\kappa_{\nu}^{-1}},
\end{equation*}
the bound~\cref{boundangle} is equivalent to $\chi\leq 1$. From the numerical results in \cref{tabLapevp}, one observes that $\chi$ is significantly smaller than $1$, demonstrating that the bound~\cref{boundangle} is not sharp. \cref{tabLapevp} confirms our theoretical result $\cos\varphi = \order(H)$ from~\cref{est_dd}. For the AGMG preconditioner, it can be observed that $\cos^{2}\varphi = \order(h)$, a very favorable behavior.

\begin{table}
    \centering
    \caption{Behavior of $\varphi$ for Laplacian eigenvalue problems with AGMG and DDM preconditioners.}
    \label{tabLapevp}

    \begin{tabular}{lccccc}
        \toprule
        \multicolumn{6}{c}{AGMG} \\ \midrule
        $h$ & $2^{-6}$ & $2^{-7}$ & $2^{-8}$ & $2^{-9}$ &$2^{-10}$  \\ \midrule
        $\cos^{2}\varphi$ & 0.0331 & 0.0192 & 0.0117 & 0.0064 & 0.0033\\ 
        $1-\kappa_{\nu}^{-1}$ & 0.6200 & 0.6260 & 0.6352 & 0.6378 & 0.6390\\
        $\chi$ & 0.0534 & 0.0307 & 0.0184 & 0.0101 & 0.0051  \\ \midrule
        \multicolumn{6}{c}{DDM with $H=2^{-2}$} \\ \midrule
        $h$ & $2^{-4}$ & $2^{-5}$ & $2^{-6}$ & $2^{-7}$ &$2^{-8}$ \\ \midrule
        $\cos^{2}\varphi$ &  0.1961 & 0.1935 & 0.1915 & 0.1905 & 0.1901\\
        $1-\kappa_{\nu}^{-1}$ & 0.8221 & 0.8201 & 0.8189 & 0.8182 & 0.8178 \\
        $\chi$ & 0.2386 & 0.2360 & 0.2339 & 0.2328 & 0.2324 \\ \midrule
        \multicolumn{6}{c}{DDM with $h=2^{-8}$} \\ \midrule
        $H$ & $2^{-2}$ & $2^{-3}$ & $2^{-4}$ & $2^{-5}$ &$2^{-6}$ \\ \midrule
        $\cos^{2}\varphi$ &  0.1901 & 0.0720 & 0.0202 & 0.0052 & 0.0013\\ 
        $1-\kappa_{\nu}^{-1}$ & 0.8178 & 0.8213 & 0.8242 & 0.8278 & 0.8320 \\ 
        $\chi$ & 0.2324 & 0.0877 & 0.0246 & 0.0063 & 0.0016 \\ \bottomrule
    \end{tabular}
    
\end{table}

\subsubsection{Empirical probability tests}

In most practical situations, PINVIT is used with a random initial vector $u_0$. Therefore it is of interest to measure the empirical success probability for our condition $\dist_{B}(u_{0},u_{*})< \varphi$ and for the condition $\lambda(u_{0})<\lambda_{2}$ required by~\cite{knyazev2003geometric}.

It is tempting to choose a Gaussian random vector $u_0$, but such a choice is unfortunate---it yields an empirical success probability close to zero for both conditions.
A Gaussian random vector tends to be highly oscillatory, whereas the eigenvector $u_{*}$ is typically very smooth.
We address this issue by using a smoother multivariate normal random vector. As the inverse Laplacian effects smoothing, it makes sense to choose $u_0 \sim \mathcal{N}(0,B^{-2})$, which can be computed as $u_0 = B^{-1} \omega$ for a Gaussian random vector $\omega$.
{Using $1000$ independent random trials, we report the empirical success probabilities in \cref{tabLapevpEP}, which} impressively show the superiority of our condition on the initial vector.

\begin{table}
    \centering
    \caption{Empirical success probabilities for Laplacian eigenvalue problems with AGMG and DDM preconditioners.}
    \label{tabLapevpEP}

    \begin{tabular}{lccccc}
        \toprule
        \multicolumn{6}{c}{AGMG} \\ \midrule
        $h$ & $2^{-6}$ & $2^{-7}$ & $2^{-8}$ & $2^{-9}$ &$2^{-10}$  \\ \midrule
        $\lambda(u_{0})<\lambda_{2}$ & 0\% & 0\% & 0\% & 0\% & 0\%\\ 
        $\dist_{B}(u_{0},u_{*})<\varphi$ & 44.8\% & 53.7\% & 62.1\% & 67.9\% &77.3\% \\ \midrule
        \multicolumn{6}{c}{DDM with $H=2^{-2}$} \\ \midrule
        $h$ & $2^{-4}$ & $2^{-5}$ & $2^{-6}$ & $2^{-7}$ &$2^{-8}$ \\ \midrule
        $\lambda(u_{0})<\lambda_{2}$ & 0.5\% & 0\% & 0\% & 0\% & 0\%\\ 
        $\dist_{B}(u_{0},u_{*})<\varphi$ & 16.9\% & 7.2\% & 4.8\% & 1.3\% & 1.2\%\\ \midrule
        \multicolumn{6}{c}{DDM with $h=2^{-8}$} \\ \midrule
        $H$ & $2^{-2}$ & $2^{-3}$ & $2^{-4}$ & $2^{-5}$ &$2^{-6}$ \\ \midrule
        $\lambda(u_{0})<\lambda_{2}$ & 0\% & 0\% & 0\% & 0\% & 0\%\\ 
        $\dist_{B}(u_{0},u_{*})<\varphi$ & 0.9\% & 11.1\% & 41.1\% & 71.3\% & 85.8\%\\ \bottomrule
    \end{tabular}

\end{table}

\subsection{Mixed-precision preconditioners for kernel matrices}

Following the setting in \cite[Section~5.4]{Kressner2023}, we perform experiments with the mixed-precision preconditioner from \Cref{sec:mixed} for targetting the smallest eigenvalues of a kernel matrix. Choosing independent Gaussian random 
vectors $x_{1},\dotsc,x_{n}\in\R^{n}$, we consider the Laplacian kernel matrix defined by
\begin{equation*}
    (A)_{ij} = \exp\Bigl(\frac{-\norm{x_{i}-x_{j}}}{2}\Bigr), \quad i,j = 1,\ldots,n.
\end{equation*}
Similarly, choosing another set of independent Gaussian random vectors $y_{1},\dotsc,y_{n}\in\R^{n}$ and 
$
    K(x,y)=(x^{\Ttran}y+1)^{3},    
$
we consider the complex kernel matrix defined by 
\begin{equation*}
    (A)_{ij} = K(x_{i},x_{j})+K(y_{i},y_{j})+\Im\bigl(K(x_{i},y_{j})-K(y_{i},x_{j})\bigr).
\end{equation*}
In both cases, we choose $B$ be the preconditioner obtained from performing the Cholesky factorization of $A$ in single precision.
As in the previous section, we measured the empirical success probability for $\dist_{B}(u_{0},u_{*})< \varphi$ and $\lambda(u_{0})<\lambda_{2}$.
We choose $u_{0}$ to be a Gaussian random vector, set $n\in\{512,1024,2048,4096\}$. {For each $n$, we verify the initial conditions on $u_{0}$ by sampling $1000$ independent random initial vectors, and collect the results in \cref{tabMP}.} With such effective mixed-precision preconditioners, our condition on the initial vector achieves nearly $100\%$ success probability, whereas the condition $\lambda(u_{0}) < \lambda_{2}$ appears to be never satisfied.

\begin{table}
    \caption{Empirical success probabilities for dense kernel matrices with mixed-precision preconditioner.}
    \label{tabMP}
    \centering
    \begin{tabular*}{\textwidth}{@{\extracolsep\fill}lcccccccc}
        \toprule
        & \multicolumn{4}{c}{Laplacian Kernel} & \multicolumn{4}{c}{Complex Kernel} \\
        \cmidrule{2-5} \cmidrule{6-9}
        $n$ & 512 & 1024 & 2048 & 4096 & 512 & 1024 & 2048 & 4096  \\ \midrule
        $\lambda(u_{0})<\lambda_{2}$ & 0\% & 0\% & 0\% & 0\% & 0\% & 0\% & 0\% & 0\%\\ 
        $\dist_{B}(u_{0},u_{*})<\varphi$ & 100\% & 100\% & 100\% & 100\% & 96.8\% & 97.9\% & 100\% & 100\% \\ 
        \bottomrule
    \end{tabular*}
\end{table}

\section{Conclusions}

In this work, we have shown that a variant of PINVIT enjoys a global convergence guarantee (formulated in terms of the angle of distortion) that is observed to be significantly better than the one required by Neymeyr's classical analysis of PINVIT in case of two classes of preconditioners. The tools developed in this work, especially the connection to Riemannian steepest descent, are potentially useful for developing and analyzing extensions, including accelerated methods in the spirit of~\cite{alimisis2024,knyazev2001toward,shao2023riemannian,Stathopoulos2010}.
\section*{Statements and Declarations}
\begin{itemize}
	\item Funding: FA was supported by the SNSF under research project 192363, BV was supported in part by the SNSF under research project 192129.
	\item Author contributions: All authors contributed equally to the paper.
	\item Code availability: Scripts to reproduce numerical results are publicly available at \url{https://github.com/nShao678/Improved-convergence-of-PINVIT}.
	\item Conflict of interests: Not applicable.
\end{itemize}

\bibliography{ref}

\end{document}